\documentclass[reqno,11pt]{amsart}
\usepackage{amsmath,amsfonts,amssymb,amsxtra,latexsym,amscd,enumerate,amsthm,verbatim}
\allowdisplaybreaks
\usepackage{hyperref}
\usepackage{enumitem}
\usepackage{graphicx}
\usepackage{extarrows}
\usepackage{tikz}
\usetikzlibrary{arrows.meta}

\hypersetup{
    colorlinks=true,
    citecolor= blue,
    linkcolor=black,
    filecolor=magenta,
    urlcolor=cyan
    }

\usepackage[margin=1in]{geometry}
\numberwithin{equation}{section}

\usepackage{subcaption}
\newtheorem{thm}{Theorem}[section]

\newtheorem{theorem}[thm]{Theorem}

\newtheorem{lemma}[thm]{Lemma}
\newtheorem{proposition}[thm]{Proposition}

\theoremstyle{definition}

\newtheorem{assumption}[thm]{Assumption}
\newtheorem{definition}[thm]{Definition}

\def\p{\partial}

\def\R{\mathbb{R}}

\theoremstyle{remark}
\newtheorem{rem}{Remark}[section]
\newtheorem{remark}[rem]{Remark}

\begin{document}
	
	\title{On Cauchy Problems for Parabolic Equations with Rough Coefficients}
	
	\author{Cheng Yuan}
\thanks{School of Mathematical Sciences; LMNS and Shanghai Key Laboratory for Contemporary Applied Mathematics, Fudan University, Shanghai 200433, P. R. China.}%
\thanks{Email: cyuan22@m.fudan.edu.cn}%
\thanks{Supported in part by NSFC 123B2008}
	
\begin{abstract}
We study Cauchy problems for parabolic equations in
both non-divergence and divergence form whose diffusion coefficients
may be discontinuous and unbounded. Under critical weak Sobolev
assumptions on the coefficients, we establish unique solvability and
optimal fractional Sobolev estimates for general prescribed, possibly
nonzero, initial data. In contrast to the homogeneous temporal trace
usually adopted in the classical BMO solvability theory, the
initial data contribute explicitly to our sharp energy estimates. The
proof combines an approximation scheme preserving the required
high-frequency smallness with paraproduct estimates, a refined critical
Sobolev embedding, and Fefferman--Stein maximal inequalities. We
further construct explicit counterexamples showing that the restriction
$s<\frac12$ is sharp.
\end{abstract}
\maketitle
\setcounter{tocdepth}{1}
\pagestyle{plain}
\tableofcontents
\section{Introduction}
Parabolic differential equations arise in most physical and financial models, and the mathematical theory of various parabolic equations is also significant. In this paper, we consider the parabolic equations in both divergence and non-divergence form with rough coefficients:
\begin{equation}\label{heatequationnondiv}
\begin{cases}
\partial_{t}u-a^{ij}(t,x)\partial_{x_{i}x_{j}}^{2}u+b^{j}(t,x)\partial_{x_{j}}u+c(t,x)u=f(t,x),\quad 0<t\leq T, x\in\mathbb{R}^{d};\\
u(0,x)=u_{0}(x),\quad x\in\mathbb{R}^{d}.
\end{cases}
\end{equation}
and
\begin{equation}\label{heatequationdiv}
\begin{cases}
\partial_{t}u-\p_{x_{i}}(a^{ij}(t,x)\p_{x_{j}}u)+b^{j}(t,x)\partial_{x_{j}}u+c(t,x)u=f(t,x),\quad 0<t\leq T, x\in\mathbb{R}^{d};\\
u(0,x)=u_{0}(x),\quad x\in\mathbb{R}^{d}.
\end{cases}
\end{equation}
In above, $ (a^{ij}(t,x))_{d\times d}$ are the diffusion coefficient, and $ f(t,x) $ denotes the external force.
\subsection{Background}\label{background}
Parabolic equations arise commonly in many natural situations, such as (heat, electric) conduction, diffusion process, and viscous medium. Moreover, the models with variable coefficients are 
significant to describe the phenomenon of time-dependence and inhomogeneity. For instance, the density or the temperature may depend on time and space. 

In our paper, we study the parabolic equations with rough coefficients, and the coefficients are even not necessarily continuous. This type of parabolic equations are also common in various important fields in physics and finance. We present several application situations:

\textit{Material Science.} We consider composite material. In the shear equation of the composite material, different shear moduli between fibers and surrounding matrix leads to the \emph{jumping} of the diffusion coefficient in the equation. Similar settings also arise in the electric conductivity problem with embedded domains. There are several works consider the gradient of the solution (representing the stress field or electric field) and related regularity theory. We refer to \cite{BASL99} for numerical result, and \cite{LV00},\cite{LN03},\cite{FKNN13},\cite{DLY23},\cite{DX26} (also the reference therein) for mathematical analysis.

\textit{Finance Pricing.} Based on Merton's structural method \cite{M74} and the celebrated Black-Scholes pricing model \cite{BS73}, the value of one corporate (zero-coupon) bond at time $t$ and firm value $S$ satisfies an 1-D parabolic equation, while the diffusion coefficients are different constants in high and low rating regions (corresponding to the risk of default), which leads to the \emph{jumping} of the coefficient when the grade of the firm changes. The solvability  and regularity theories of the equation are important, toward the equivalence between the PDE and the original BSDE from the structural model.  For the detailed background and derivation of these models and existing mathematical results, we refer to the books \cite{BR02,J05,HL24} and the reference therein.  The mathematics theory of parabolic equations with jumping coefficient is also closely tied to the theory of stochastic processes, see e.g. \cite{K04}.

\textit{Viscous Fluid.} Consider the compressible Navier-Stokes equations in Lagrangian coordinates, it is a coupled system with both conservation law structure and diffusion structure. More precisely, the diffusion coefficient of the velocity depends on the specific volume. In the rough regularity setting, the specific volume may exhibit to be a shock, which leads to the discontinuity of the diffusion coefficient in the equation of velocity. We refer to the monograph \cite{S94} for the detailed description of related models, and \cite{LY22,WYZ22,WYZ24,CHN24,CLN23} for the recent progress on the well-posedness theory of compressible Navier-Stokes equations under rough regularity.

In this paper, we consider the following parabolic equations in the critical regularity settings:
\begin{equation}\label{heateqnondiv}
\begin{cases}
\partial_{t}u-a(t,x)\partial_{x}^{2}u=f(t,x),\quad 0<t\leq T, x\in\mathbb{R},\\
u(0,x)=u_{0}(x),\quad x\in\mathbb{R},
\end{cases}
\end{equation}
and
\begin{equation}\label{heateqdiv}
\begin{cases}
\partial_{t}u-\p_{x}(a(t,x)\p_{x}u)=f(t,x),\quad 0<t\leq T, x\in\mathbb{R},\\
u(0,x)=u_{0}(x),\quad x\in\mathbb{R}.
\end{cases}
\end{equation}
We consider the above 1-D parabolic equations without lower order terms for clarity. The extension to higher dimensions
is not merely notational and requires several modifications. In this paper, we will study the solvability and regularity theory for Cauchy problem \eqref{heateqnondiv} and \eqref{heateqdiv} under the critical regularity of the diffusion coefficient $a(t,x)$.
\subsection{Main Results}\label{mainresults}
We first state our main assumptions on the diffusion coefficient $a(t,x)$.
\begin{assumption}\label{mainassume}
Assume the diffusion coefficient $a(t,x)$ satisfies the following hypotheses (see \eqref{nonlocaldef} for the definition of $|\p_{x}|^{s})$:
\leavevmode
\begin{itemize}[
  leftmargin=*,
  widest={$\mathrm{(H3)}_{\delta,s,D}$},
  labelsep=0.6em
]

\item[(H1)]
$a(t,x)$ is measurable and uniformly elliptic:
\begin{equation}
a(t,x)\geq c_{0}>0,
\qquad
\forall (t,x)\in[0,T]\times\mathbb{R}.
\end{equation}

\item[(H2)]
For a.e. $t\in[0,T]$, $a(t,\cdot)\in\mathcal{S}'_{x}(\mathbb{R})$, and
\begin{equation}
\sup_{t\in[0,T]}
\left|
\langle a(t,\cdot),\phi\rangle_{\mathcal{S}'_{x}\times\mathcal{S}_{x}}
\right|
\leq C_{\phi},
\qquad
\forall\phi\in\mathcal{S}(\mathbb{R}).
\end{equation}

\item[$\mathrm{(H3)}_{\delta,s,D}$]
Let $S_N$ be the Littlewood--Paley projection operator onto low
frequencies $|\xi|\lesssim 2^N$ (see Subsection~\ref{littlewoodpaley}
for details). Then
\begin{equation}\label{regularityassume}
\sup_{t\in[0,T]}
\left(
 \|\partial_x S_0a\|_{L_x^{1/s,\infty}}
 +
 \bigl\||\partial_x|^s(\operatorname{Id}-S_0)a
 \bigr\|_{L_x^{1/s,\infty}}
\right)
\leq D.
\end{equation}
Moreover, there exists $N\in\mathbb{N}$ such that
\begin{equation}\label{smallassume}
\sup_{t\in[0,T]}
\bigl\|
|\partial_x|^s(\operatorname{Id}-S_N)a(t,\cdot)
\bigr\|_{L_x^{1/s,\infty}}
\leq\delta.
\end{equation}

\end{itemize}
\end{assumption}

Now we state our main results as follow:
\begin{theorem}\label{main} For any $ 0<s<\frac{1}{2}$, there exists a universal constant $\delta_{0}=\delta_{0}(c_{0},s,D)$, such that: under the assumptions (H1), (H2), $(H3)_{\delta_{0},s,D}$, for any continuous initial data $u_{0}(x)$ with $u_{0}(x)\in \dot{H}^{1}(\R)\cap \dot{H}^{1+s}(\R)$ and force term $f(t,x)\in L^{2}([0,T];H^{s}(\R))$, the non-divergence form equation \eqref{heateqnondiv} exists a unique strong solution $u(t,x)$, which satisfies
\begin{equation}\label{nondivenergy}
\sup_{0\leq t \leq T}\|\partial_{x}u(x,t)\|_{H^{s}_{x}}^{2}+\int_{0}^{T}\|\partial_{x}^{2}u\|_{H^{s}_{x}}^{2}dt\lesssim \|u_{0}(x)\|_{\dot{H}^{1}\cap \dot{H}^{1+s}(\R)}^{2}+\|f(t,x)\|_{L^{2}([0,T];H^{s}(\R))}^{2}.
\end{equation}
\end{theorem}
\begin{theorem}\label{maintwo}For any $ 0<s<\frac{1}{2}$, there exists a universal constant $\delta_{0}=\delta_{0}(c_{0},s,D)$, such that: under the assumptions (H1), (H2), $(H3)_{\delta_{0},s,D}$, for any initial data $u_{0}(x)$ with $u_{0}(x)\in H^{s}(\R)$ and force term $f(t,x)\in L^{2}([0,T];H^{s-1}(\R))$, the divergence form equation \eqref{heateqdiv} exists a unique weak solution $u(t,x)$, which satisfies
\begin{equation}\label{divenergy}
\sup_{0\leq t \leq T}\|u(x,t)\|_{H^{s}_{x}}^{2}+\int_{0}^{T}\|\partial_{x}u\|_{H^{s}_{x}}^{2}dt\lesssim \|u_{0}(x)\|_{H^{s}(\R)}^{2}+\|f(t,x)\|_{L^{2}([0,T];H^{s-1}(\R))}^{2}.
\end{equation}
\end{theorem}
We present several comments for our main assumptions and results:
\begin{remark}
We can see that both the energy estimates \eqref{nondivenergy} and \eqref{divenergy} coincide with the standard energy estimates for the heat equation. 
\end{remark}
\begin{remark} For a fixed $s\in(0,\frac{1}{2})$, to obtain the $H^{s}$ regularity of $\p_{x}^{2}u$ in \eqref{heateqnondiv} or $\p_{x}u$ in \eqref{heateqdiv}, the requirement that $a(t,x)$ should carry at least $|\p_{x}|^{s}$  derivative is reasonable in the sense of paraproduct. 

Indeed, the regularity assumption on $a(t,x)$ is also critical in the sense of natural parabolic scaling $a(t,x)\rightarrow a_{\lambda}(t,x):=a(\lambda^{2}t,\lambda x)$. Under subcritical regularity assumption, say e.g. $a(t,x)$ can be embedded into $C^{\epsilon}_{x}$, then the well-known $C^{2,\alpha}$ and $W^{2,p}$ theories can be applied to the heat operator $\p_{t}-a\p_{xx}$, see \cite{L67}. Moreover, whenever the coefficient is H\"older continuous,  the heat operator $\p_{t}-a\p_{xx}$ admits a fundamental solution by Levi's parametrix method (see \cite{F83}) and we can study the corresponding solution directly from the explicit expression.
\end{remark}
\begin{remark}The assumption on the low-frequency part of $a(t,x)$ is much weaker such that our results can be applied to large range of diffusion coefficients that may grow as $|x|\to\infty$. The smallness requirement in \eqref{smallassume} is also weaker than the usual sense. Indeed, we only require smallness outside an \emph{arbitrary given} frequency.
\end{remark}
\begin{remark}
Under our assumptions, the coefficient $a(t,x)$ is allowed to be discontinuous or unbounded thus many classical theories cannot be applied. For example, 
\begin{equation}\label{stepfunction}
a_{1}(x):=
\begin{cases}
\sigma_{1},\ x<0,\\
\sigma_{2},\ x\geq 0,
\end{cases}
\end{equation}
and
\begin{equation}\label{logfunction}
a_{2}(x)=\lambda\cdot\log(x^{2}+1)+1
\end{equation}
fulfills Assumption \ref{mainassume} for suitable $\sigma_{1},\sigma_{2},\lambda>0$ and $\sigma_{1}\neq \sigma_{2}$.
\end{remark}

\begin{remark} In the regularity assumption \eqref{regularityassume}, the underlying Lorentz space $L^{\frac{1}{s},\infty}$ is the weakest one in same scaling, and it contains all critical scaling Besov spaces with regularity index larger than $s$ (see \cite{JJ76} or \cite{T83}):
\begin{equation}
\||\p_{x}|^{s}(\mathrm{Id}-S_{0})h\|_{L^{\frac{1}{s},\infty}}\lesssim_{s,t}\sup_{j\geq 0}2^{jt}\|\Delta_{j}h\|_{L^{\frac{1}{t},\infty}},\ \forall 0<s<t<1.
\end{equation}
\end{remark}

We emphasize that the conclusions in our main results are optimal, in the sense that it can not be further refined \emph{even for smooth and compactly supported initial data and zero force}.\\

\textbf{Optimality.} We will illustrate that the threshold $s<\frac{1}{2}$ is optimal. It means that we can not expect higher regularity of the solution even for smooth and compact-supported initial data.

We illustrate with a typical and significant setting, which commonly arise in financial pricing models: the diffusion coefficient is a step function, i.e. we consider \eqref{heateqnondiv} with coefficient $a(t,x)=a_{1}(x)$, where $a_{1}(x)$ defined in \eqref{stepfunction} with $\sigma_{1}>0,\sigma_{2}>0$ and $\sigma_{1}\neq \sigma_{2}$.  Without loss of generality, we assume $\sigma_{1}=1,\sigma_{2}=2$. The initial data $u_{0}(x)$ is chosen to be an arbitrary smooth function and compactly supported. The force is chosen to be zero.

Assume \eqref{heateqnondiv} admits a unique strong solution $u(t,x)$ satisfies \eqref{nondivenergy} for $s = \frac{1}{2}$. Using a uniqueness argument we can obtain that
\begin{itemize}
\item[(i)] For $x<0$:
\begin{equation}\label{leftsol}
\begin{split}
u(t,x)&=\int_{0}^{\infty}\left[(2-\sqrt{2})\frac{1}{\sqrt{4\pi t}}\exp\left(-\frac{(\sqrt{2}x-y)^{2}}{8t}\right)\right]u_{0}(y)dy\\
+&\int_{-\infty}^{0}\left[\frac{1}{\sqrt{4\pi t}}\exp\left(-\frac{(x-y)^{2}}{4t}\right)+(3-2\sqrt{2})\frac{1}{\sqrt{4\pi t}}\exp\left(-\frac{(x+y)^{2}}{4t}\right)\right]u_{0}(y)dy;
\end{split}
\end{equation}
\item[(ii)] For $x>0$:
\begin{equation}\label{rightsol}
\begin{split}
u(t,x)&=\int_{-\infty}^{0}\left[(4\sqrt{2}-4)\frac{1}{\sqrt{8\pi t}}\exp\left(-\frac{(x-\sqrt{2}y)^{2}}{8t}\right)\right]u_{0}(y)dy\\
+&\int_{0}^{\infty}\left[\frac{1}{\sqrt{8\pi t}}\exp\left(-\frac{(x-y)^{2}}{8t}\right)+(2\sqrt{2}-3)\frac{1}{\sqrt{8\pi t}}\exp\left(-\frac{(x+y)^{2}}{8t}\right)\right]u_{0}(y)dy.
\end{split}
\end{equation}
\end{itemize}
Since $\p_{x}^{2}u(t,x)\in L^{2}_{t}H^{\frac{1}{2}}_{x}$ by our hypotheses, it must hold $\p_{x}^{2}u(t,0-)=\p_{x}^{2}u(t,0+)$ for almost all $t>0$. However, by direct calculation from \eqref{leftsol} and \eqref{rightsol} , it fails \emph{unless}
\begin{equation}
    \sqrt{2}u_{0}''(\sqrt{2}z)+u_{0}''(-z)\equiv 0,\ \forall z\geq 0.
\end{equation}

\subsection{Literature review}
In this subsection, we review some earlier results focus on the regularity theory of \eqref{heatequationnondiv} and \eqref{heatequationdiv}. Due to the vast literature, we focus only on those most relevant to our work.

The regularity theory on \eqref{heatequationnondiv} and \eqref{heatequationdiv} with continuous diffusion coefficient $a(t,x)$ is classical. For H\"older continuous diffusion coefficient $a(t,x)$, the celebrated Schauder estimates consider the $C^{2,\alpha}_{x}$ estimate for \eqref{heatequationnondiv} and $C^{1,\alpha}_{x}$ estimates for \eqref{heatequationdiv}, respectively. For the diffusion coefficient $a(t,x)$ is merely bounded and continuous, the $W^{2,p}_{x}$ estimate for \eqref{heatequationnondiv} and $W^{1,p}_{x}$ estimates for \eqref{heatequationdiv} can be also obtained respectively. We refer to \cite{L67} and the reference therein.

For merely bounded and measurable diffusion coefficients $a(t,x)$, the celebrated De Giorgi--Nash--Moser theory implies that the weak solution of the divergence-form equation \eqref{heatequationdiv} is indeed H\"older continuous, toward Hilbert's 19th problem. We refer to the original literature \cite{D57,N58,M60} and the recent note \cite{V16}. For \eqref{heatequationnondiv} with merely bounded and measurable diffusion coefficient, the H\"older continuity of the solution was proved by the seminal work due to Krylov and Safonov \cite{KS80}.

A major question is that whether we can obtain the higher regularity of the solution \emph{with discontinuous diffusion coefficient}. Without continuity, classical regularity theory (for e.g. Schauder estimates, $W^{1,2}_{p}$ estimates) can not be applied. However whether the $W^{1,2}_{p}$ theory is valid in rough setting is important since there are solutions of stochastic differential equations corresponding to the parabolic operator via Ito's lemma, while Ito's lemma is applicable only for the solution $u\in W^{1,2}_{p}$ for some large $p$ depends on the dimension (see Remark 2.2 in \cite{K07}).

Initially, the $W^{1,2}_{p}$ theory for \eqref{heatequationnondiv} was studied by Bramanti and Cerutti \cite{BC93}, under the assumptions of $a(t,x)\in VMO_{t,x}$ (the space of functions of vanishing mean oscillation). The proof of \cite{BC93} is based on explicit expression of the solution and careful estimates of singular integral operators. We refer to the book \cite{Book00} for a detailed review. Subsequently, the vanishing mean oscillation condition $VMO_{t,x}$ for all $(t,x)$ variables was relaxed to $VMO_{x}$ by Krylov \cite{K07}. In \cite{K07}, the author stated the $W^{1,2}_{p}$ theory for \eqref{heatequationnondiv} and $H^{1}_{p}$ theory for \eqref{heatequationdiv} with $a(t,x)\in VMO_{x}$, the proof is based on delicate pointwise estimates of associated Fefferman-Stein sharp functions. The $W^{1,2}_{p}$ theory was subsequently extended for partially vanishing mean oscillation coefficient, i.e. we do not need the VMO condition for all spatial variables. We refer to \cite{KK07} for $W^{1,2}_{p}$ theory of non-divergence form equation \eqref{heatequationnondiv} and \cite{D10} for $H^{1}_{p}$ theory divergence form equation \eqref{heatequationdiv}, with partially VMO (or small BMO) coefficients. In recent years there are also further extensions along this direction, we refer to the review paper \cite{D20} and the reference therein.

In this work, we study the Cauchy problems with rough coefficient with critical regularity. We expect that our argument can be applied to further related models. We also point that our results also imply the $W^{1,2}_{p}$ bounds of the solution for some suitable $p=p(s)$ by interpolation between $L^{2}_{t}H^{2+s}_{x}$ and $H^{1}_{t}H^{s-}_{x}$ (see Remark \ref{remarkfinal}).

\subsection{Outline of the proof}
In this subsection we briefly present the main steps of our argument.

\textbf{Step 1: Approximation.} To solve the original equations \eqref{heateqnondiv} and \eqref{heateqdiv}, we first construct approximation systems (see \eqref{approxnondiv} and \eqref{approxdiv}) such that we can uniquely solve them to obtain classical approximating solutions. To this end, we modify the rough coefficient $a(t,x)$ to obtain a smooth and bounded function $a^{\epsilon}(t,x)$. Since $a(t,x)$ may be unbounded in our assumptions, the construction of $a^{\epsilon}(t,x)$ is followed from two steps: (i) convolution with smooth mollifier; (ii) truncating the unbounded part of $a(t,x)$. An important fact is that $a^{\epsilon}$ still uniformly preserve the high-frequency smallness, which can be obtained by a careful para-linearized estimates, as well as Fefferman--Stein maximal inequalities. We refer to Section \ref{approximation} for details.

\textbf{Step 2: Non-local energy estimates.} In this step, we need to obtain uniformly-in-$\epsilon$ non-local energy estimates for the approximation systems \eqref{approxnondiv} and \eqref{approxdiv}, toward \eqref{nondivenergy} and \eqref{divenergy}, respectively. Roughly speaking, we can see that it is essential to estimate the following commutator (we denote $[A,B]:=AB-BA$)
\begin{equation}\label{commut}
\int_{\R}\left([S_{>\Lambda}|\p_{x}|^{s},a^{\epsilon}]\p_{x}^{2}u^{\epsilon}\right)\cdot \left(S_{>\Lambda}|\p_{x}|^{2+s}u^{\epsilon}\right)dx,
\end{equation}
where $S_{>\Lambda}$ is the projection operator to the high frequency $\gtrsim 2^{\Lambda}$ for some borderline frequency scale $\Lambda\gg 1$. To estimate the non-local commutator \eqref{commut}, we use paraproduct decomposition into dyadic frequency blocks (see \eqref{commutator}), and employing delicate analytic tools such as refined Sobolev embeddings and Fefferman--Stein maximal inequalities to close the estimates.
\begin{remark} We point that the study of commutator estimate involved non-local operators can be traced back to Kato and Ponce \cite{KP88}. See also the refinements in \cite{Li19}.
\end{remark}
\textbf{Step 3: Passing to the limit.} After obtaining the uniform-in-$\epsilon$ estimates for the approximation systems, by using (weak) compactness we then pass to the limit to obtain the unique solvability of original systems \eqref{heateqnondiv} and \eqref{heateqdiv} in optimal fractional Sobolev spaces. Since $a(t,x)$ may unbounded in our assumptions, we also need careful localization procedures in the argument of passing to the limit and proving the uniqueness. We refer to Section \ref{pfmain} for details.
\subsection{Structure of paper} The rest of the paper is structured as follow.

In Section \ref{pre}, we present various analysis tools that were used through the article, including Littlewood-Paley theory, Hardy-Littlewood maximal function, Fefferman--Stein maximal inequalities.

In Section \ref{approximation}, we construct the approximation systems. We also show that the approximating diffusion coefficient $a^{\epsilon}$ uniformly preserve the high-frequency smallness.

In Section \ref{prioriesti}, we derive the uniform-in-$\epsilon$ non-local energy estimates in optimal fractional Sobolev spaces  for approximating system.

In Section \ref{pfmain}, we pass to the limit the obtain the unique solvability of the original equations in optimal fractional Sobolev spaces.

\emph{Notations.} In this article, we adopt the following notation conventions:
\begin{itemize}
\item[(1)] We use $ A\lesssim B $ means $ A\leq c B $ for some absolute constant $ c $. The notation $ A\sim B $ means $ A\lesssim B $ and $ B\lesssim A $. Additionally, $ A \lesssim_{\star} B $ implies $ A\leq c B $ with $ c $ dependent on a specific quantity $ \star $, i.e., $ c=c(\star) $. 

\item[(2)] For $1\leq p<\infty, 1\leq q\leq \infty$, the Lorentz space $L^{p,q}(\R^{d})$ is defined by
\begin{equation*}
\|u\|_{L^{p,q}}=
\begin{cases}
\displaystyle\left(\int_{0}^{\infty}\left(\lambda [d_{u}(\lambda)]^{\frac{1}{p}}\right)^{q}\frac{d\lambda}{\lambda}\right)^{\frac{1}{q}},\quad q<\infty,\\
 \displaystyle\sup_{\lambda\geq 0}\left(\lambda [d_{u}(\lambda)]^{\frac{1}{p}}\right)\quad q=\infty,
\end{cases}
\end{equation*}
where $d_{u}$ is the distribution function for measurable function $u(x)$
\begin{equation*}
d_{u}(\lambda)=\left|\left\{x\in\R^{d}:|u(x)|>\lambda\right\}\right|.
\end{equation*}
In particular, $L^{p,p}$ is equal to the standard Lebesgue space $L^{p}$. We also adopt the modification that $L^{\infty,\infty}=L^{\infty}$.

\item[(3)] $|\partial_{x}|^{s}$ denotes the non--local differential operator with respect to $x$, defined using the Fourier transform as
\begin{equation}\label{nonlocaldef}
    |\partial_{x}|^{s}f:=\mathcal{F}_{x}^{-1}(|\xi|^{s}\mathcal{F}_{x}(f)).
\end{equation}
\end{itemize}

\section{Preliminaries}\label{pre}
In this section, we present some well-known analytic tools, which will be frequently employed.
\subsection{Littlewood--Paley Decomposition}\label{littlewoodpaley}
First we recall the definition of Sobolev spaces.
\begin{definition}For $s\in\mathbb{R}$, the standard homogeneous Sobolev spaces $\dot{H}^{s}$ and inhomogeneous Sobolev spaces $H^{s}(\R^{d})$ are defined by
\begin{equation}
\begin{split}
&\dot{H}^{s}(\mathbb{R}^{d}):=\left\{u\in\mathcal{S}':\|u\|_{\dot{H}^{s}(\R^{d})}^{2}:=\int_{\R^{d}}|\xi|^{2s}|\hat{u}(\xi)|^{2}d\xi<\infty\right\},\\
&H^{s}(\mathbb{R}^{d}):=\left\{u\in\mathcal{S}':\|u\|_{H^{s}(\R^{d})}^{2}:=\int_{\R^{d}}(1+|\xi|^{2})^{s}|\hat{u}(\xi)|^{2}d\xi<\infty\right\}.
\end{split}
\end{equation}
\end{definition}

We also need the Littlewood--Paley decomposition, for more details we refer to \cite{BCD}. Let
\begin{equation*}
\mathcal{C}=\left\{ \xi \in \R^{d} :\frac{3}{4}\leq |\xi|\leq \frac{8}{3}\right \},\quad \mathcal{D}=\left\{ \xi \in \R^{d}: |\xi|\leq \frac{4}{3}\right\}.
\end{equation*}
There exists $\varphi \in C_{c}^{\infty}(\mathcal{C})$, $\chi \in C_{c}^{\infty}(\mathcal{D})$, such that
\begin{equation*}
\forall \xi\in\R^{d}\ ,\ \chi(\xi)+\sum_{j\geq 0}\varphi(2^{-j}\xi)=1,
\end{equation*}
For $u\in\mathcal{S}'$, define the Littlewood-Paley projection operators
\begin{equation*}
S_{0}u=\mathcal{F}^{-1}\left(\chi(\xi)\hat{u}(\xi)\right),\quad \Delta_{j} u=\mathcal{F}^{-1}\left(\varphi(2^{-j}\xi)\hat{u}(\xi)\right), j\geq 0.
\end{equation*}
We have the following Littlewood--Paley decomposition of $u\in\mathcal{S}'$:
\begin{equation*}
u=S_{0}u+\sum_{j\geq 0}\Delta_{j}u.
\end{equation*}
In particular, by Plancherel's identity, we have the following equivalent norms:
\begin{equation}
\|u\|_{H^{s}}^{2}\sim \|S_{0}u\|_{L^{2}}^{2}+\sum_{j\geq 0}2^{2js}\|\Delta_{j}u\|_{L^{2}}^{2};\\
\end{equation}
Furthermore, we have the following Littlewood--Paley Characteristic hold for Lorentz spaces:
\begin{lemma}\label{lpchar} For $1<p<\infty$ and $1\leq q\leq \infty$, we have
\begin{equation}
\|u\|_{L^{p,q}}\sim_{p,q} \left\|\left(|S_{0}u|^{2}+\sum_{j\geq 0}|\Delta_{j}u|^{2}\right)^{\frac{1}{2}}\right\|_{L^{p,q}}.
\end{equation}
\end{lemma}
\begin{proof}The case of $p=q$ (i.e. standard Lebesgue spaces) is classical, see \cite{GTM249}. For general Lorentz spaces, the one side inequality
\begin{equation*}
\left\|\left(|S_{0}u|^{2}+\sum_{j\geq 0}|\Delta_{j}u|^{2}\right)^{\frac{1}{2}}\right\|_{L^{p,q}}\lesssim \|u\|_{L^{p,q}}
\end{equation*}
follows from real interpolation, and the other side inequality follows from a standard duality argument.
\end{proof}
The Littlewood--Paley projection operators satisfy the following Bernstein's inequalities.
\begin{lemma}\label{bernstein}For $1\leq p\leq\infty, k \geq 0$,
\begin{equation}\label{bernsteineq}
\begin{split}
&\|\nabla^{k}(S_{0}u)\|_{L^{p}(\R^{d})}\lesssim \|S_{0}u\|_{L^{p}(\R^{d})},\\
&\|\nabla^{k}(\Delta_{j}u)\|_{L^{p}(\R^{d})} \sim 2^{kj}\|\Delta_{j}u\|_{L^{p}(\R^{d})}.
\end{split}
\end{equation}
For $1<p<q<\infty, k\geq 0$,
\begin{equation}\label{refinebernstein}
\begin{split}
&\|\nabla^{k}(S_{0}u)\|_{L^{\infty}(\R^{d})}\lesssim \|S_{0}u\|_{L^{p,\infty}(\R^{d})},\\
&\|\nabla^{k}(S_{0}u)\|_{L^{q,1}(\R^{d})}\lesssim \|S_{0}u\|_{L^{p,\infty}(\R^{d})},\\
&\|\nabla^{k}(\Delta_{j}u)\|_{L^{\infty}(\R^{d})}\lesssim 2^{(k+\frac{d}{p})j}\|\Delta_{j}u\|_{L^{p,\infty}(\R^{d})},\\
&\|\nabla^{k}(\Delta_{j}u)\|_{L^{q,1}(\R^{d})}\lesssim 2^{(k+\frac{d}{p}-\frac{d}{q})j}\|\Delta_{j}u\|_{L^{p,\infty}(\R^{d})}.
\end{split}
\end{equation}
\end{lemma}
\begin{proof}The estimates in \eqref{bernsteineq} follow directly from \cite{BCD}. For the Lorentz refinement estimates \eqref{refinebernstein}, rewrite the LHS of \eqref{refinebernstein} as convolution in physical coordinate, and then applying H\"older inequality and Young-O'Neil inequality in Lorentz spaces.
\end{proof}
To analyze the product of two functions, we also need the following \emph{Bony's paraproduct decomposition}, here we denote
\begin{equation*}
S_{j}u=S_{0}u+\sum_{0\leq k<j}\Delta_{k}u.
\end{equation*}
Therefore, we can decompose the product of two functions by
\begin{equation}\label{bony}
\begin{split}
uv&=\sum_{j\geq 1}S_{j-1}u\Delta_{j}v+\sum_{j\geq 1}S_{j-1}v\Delta_{j}u\\
&\quad\quad+S_{0}u\Delta_{0}v+S_{0}v S_{1}u+\sum_{|k-j|\leq 1;k,j\geq 0}\Delta_{k}u\Delta_{j}v.
\end{split}
\end{equation}
In particular, for $q\geq10$, we have
\begin{equation}\label{project}
\Delta_{q}(uv)=\sum_{|j-q|\leq 5}\Delta_{q}\left(S_{j-1}u\Delta_{j}v\right)+\sum_{|j-q|\leq 5}\Delta_{q}\left(S_{j-1}v\Delta_{j}u\right)+\sum_{|k-j|\leq 1, j\geq q-5}\Delta_{q}\left(\Delta_{k}u\Delta_{j}v\right).
\end{equation}
\subsection{Hardy--Littlewood Maximal Function}
We also need to use the Hardy--Littlewood maximal function.
\begin{definition}For $f \in L^{1}_{loc}(\R^{d})$, the Hardy--Littlewood maximal function $\mathcal{M}(f)(x)$ is defined by
\begin{equation*}
\mathcal{M}(f)(x)=\sup_{r>0}\frac{1}{|B(x,r)|}\int_{B(x,r)}|f(y)|dy.
\end{equation*}
\end{definition}
Clearly, $\|\mathcal{M}(f)(x)\|_{L^{\infty}}\leq \|f\|_{L^{\infty}}$. More generally, by the method of real interpolation, for $1<p<\infty, 1\leq q\leq \infty$ we have
\begin{equation}
\|\mathcal{M}(f)\|_{L^{p,q}(\R^{d})}\lesssim_{p,q} \|f\|_{L^{p,q}(\R^{d})},
\end{equation}

The following Fefferman--Stein vector maximal inequality (see \cite{FS}) is useful. For brevity we only consider the index for our applications.
\begin{lemma}\label{fs}
For $(p,q)\in(2,\infty)\times[2,\infty]$ and a sequence $\{f_{k}\}_{k\geq 0}\in L^{p,q}$, the following estimate holds:
\begin{equation}\label{lpqfs}
\left\|\left(\sum_{k\geq 0}|\mathcal{M}(f_{k})|^{2}\right)^{\frac{1}{2}}\right\|_{L^{p,q}}\lesssim \left\|\left(\sum_{k\geq 0}|f_{k}|^{2}\right)^{\frac{1}{2}}\right\|_{L^{p,q}}.
\end{equation}
\end{lemma}
\begin{proof}
    Choosing a non-negative $\phi(x)\in L^{\frac{p}{p-2},\frac{q}{q-2}}(\R^{d})$. By Lemma 1 in \cite{FS} we have
    \begin{equation}\label{keymaximal}
        \sum_{k\geq 0}\int_{\R^{d}}|\mathcal{M}(f_{k})|^{2}\phi(x)dx\lesssim\sum_{k\geq 0}\int_{\R^{d}}|f_{k}|^{2}\mathcal{M}(\phi)(x)dx.
    \end{equation}
    Hence, by duality,
    \begin{equation*}
    \begin{split}
        &\left\|\left(\sum_{k\geq 0}|\mathcal{M}(f_{k})|^{2}\right)^{\frac{1}{2}}\right\|_{L^{p,q}}^{2}\\
        &\simeq \left\|\left(\sum_{k\geq 0}|\mathcal{M}(f_{k})|^{2}\right)\right\|_{L^{\frac{p}{2},\frac{q}{2}}}\simeq \sup_{\phi\geq 0,\|\phi\|_{L^{\frac{p}{p-2},\frac{q}{q-2}}}=1 } \sum_{k\geq 0}\int_{\R^{d}}|\mathcal{M}(f_{k})|^{2}\phi(x)dx\\
        &\stackrel{(2.10)}{\lesssim}\sup_{\phi\geq 0,\|\phi\|_{L^{\frac{p}{p-2},\frac{q}{q-2}}}=1 }\sum_{k\geq 0}\int_{\R^{d}}|f_{k}|^{2}\mathcal{M}(\phi)(x)dx\\
        &\lesssim \sup_{\phi\geq 0,\|\phi\|_{L^{\frac{p}{p-2},\frac{q}{q-2}}}=1 } \left\|\left(\sum_{k\geq 0}|f_{k}|^{2}\right)\right\|_{L^{\frac{p}{2},\frac{q}{2}}}\cdot\|\mathcal{M}(\phi)\|_{L^{\frac{p}{p-2},\frac{q}{q-2}}}\lesssim \left\|\left(\sum_{k\geq 0}|f_{k}|^{2}\right)^{\frac{1}{2}}\right\|_{L^{p,q}}^{2}.
        \end{split}
    \end{equation*}
In the last inequality we invoked the fact that $\|\mathcal{M}(\phi)\|_{L^{\frac{p}{p-2},\frac{q}{q-2}}}\lesssim \|\phi\|_{L^{\frac{p}{p-2},\frac{q}{q-2}}}$, hence we conclude the proof of \eqref{lpqfs}.
\end{proof}
We next recall the following pointwise estimate for convolution operators, for details see \cite{GTM249}, section 2.1.
\begin{lemma}\label{convolution}
Assume that $K\geq 0$ is a continuous even function, and $K(x)\geq K(y)$ whenever $|x|\leq |y|$, then
\begin{equation}
\sup_{\mu>0}\left(\int_{\R}\frac{1}{\mu}K\left(\frac{y}{\mu}\right)|f(x-y)|dy\right)\lesssim \|K\|_{L^{1}(\R)}\mathcal{M}(f)(x).
\end{equation}
\end{lemma}

\section{Construction of approximation systems}\label{approximation}
First we construct the approximation systems for $0<\epsilon\leq \min\{\frac{1}{10},\frac{1}{2c_{0}}\}$.
\begin{equation}\label{approxnondiv}
\begin{cases}
\p_{t}u^{\epsilon}-a^{\epsilon}(t,x)\p_{x}^{2}u^{\epsilon}&=f^{\epsilon}(t,x),\quad 0<t\leq T, x\in\R,\\
u^{\epsilon}(0,x)=u_{0}^{\epsilon}(x),
\end{cases}
\end{equation}
and
\begin{equation}\label{approxdiv}
\begin{cases}
\p_{t}u^{\epsilon}-\p_{x}(a^{\epsilon}(t,x)\p_{x}u^{\epsilon})&=f^{\epsilon}(t,x),\quad 0<t\leq T, x\in\R,\\
u^{\epsilon}(0,x)=u_{0}^{\epsilon}(x),
\end{cases}
\end{equation}
In above $(u_{0}^{\epsilon}(x), f^{\epsilon}(t,x), a^{\epsilon}(t,x))$ is defined by
\begin{equation}\label{defapp}
\begin{split}
&u_{0}^{\epsilon}(x)=\kappa_{\epsilon}(x)\cdot (\eta_{\epsilon}*_{x}u_{0})(x),\\
&f^{\epsilon}(t,x)=\kappa_{\epsilon}(x)\cdot (\tilde{\eta}_{\epsilon}*_{t,x}f)(t,x),\\
&a^{\epsilon}(t,x)=\chi_{\epsilon}\left((\tilde{\eta}_{\epsilon}*_{t,x}a)(t,x)\right),
\end{split}
\end{equation}
where $\eta_{\epsilon}$ is the standard one-dimensional mollifier and $\tilde{\eta}_{\epsilon}$ is the standard two-dimensional mollifier. The cut-off functions $\kappa_{\epsilon}(\cdot)$ and $\chi_{\epsilon}(\cdot)$ are defined by
\begin{equation}
\kappa_{\epsilon}(z):=
\begin{cases}
1,\ |z|\leq\epsilon^{-1},\\
0,\ |z|\geq 2\epsilon^{-1},\\
\end{cases}
\quad
\chi_{\epsilon}(z):=
\begin{cases}
z,\ |z|\leq\epsilon^{-1},\\
2\epsilon^{-1},\ |z|\geq 2\epsilon^{-1},\\
\end{cases}
\end{equation}
with smoothing monotone interpolations in $\epsilon^{-1}\leq |z| \leq 2\epsilon^{-1}$, and
\begin{equation}
    |\p_{z}^{k}\kappa_{\epsilon}|\lesssim_{k} \epsilon^{k},\ |\p_{z}^{k}\chi_{\epsilon}|\lesssim_{k} \epsilon^{k-1},\ \forall\ k\geq 1.
\end{equation}
\begin{remark}The cut-off function $\chi_{\epsilon}(\cdot)$ is necessary since $a(t,x)$ may be unbounded.
\end{remark}
We have the following simple but important facts, and the proof follows from standard convolution estimates and tame estimates, hence we omit it.
\begin{lemma}\label{mollifierbound}
Assume $u_{0}(x)$ and $f(t,x)$ satisfies the assumptions in Theorem \ref{main}, then
\begin{equation}
\begin{split}
&\|u_{0}^{\epsilon}\|_{\dot{H}^{r}(\R)}\lesssim \|u_{0}\|_{\dot{H}^{1}(\R)}+\|u_{0}\|_{\dot{H}^{r}(\R)}+o_{\epsilon}(1),\ \forall r\in[1,1+s],\\
&\|f^{\epsilon}\|_{L^{2}([0,T]\times H^{t}(\R))}\lesssim \|f\|_{L^{2}([0,T]\times H^{t}(\R))}+o_{\epsilon}(1),\ \forall t\in[0,s].
\end{split}
\end{equation}
Assume $u_{0}(x)$ and $f(t,x)$ satisfies the assumptions in Theorem \ref{maintwo}, then
\begin{equation}
\begin{split}
&\|u_{0}^{\epsilon}\|_{H^{r}(\R)}\lesssim \|u_{0}\|_{H^{r}(\R)}+o_{\epsilon}(1),\ \forall r\in[0,s],\\
&\|f^{\epsilon}\|_{L^{2}([0,T]\times H^{t}(\R))}\lesssim \|f\|_{L^{2}([0,T]\times H^{t}(\R))}+o_{\epsilon}(1).\ \forall t\in [-1,s-1].
\end{split}
\end{equation}
\end{lemma}

 By the standard parabolic theory, the approximation systems \eqref{approxnondiv} and \eqref{approxdiv} admit unique smooth solutions $u^{\epsilon}(t,x)\in H^{\infty}_{t,x}$. Our task is to derive uniform-in-$\epsilon$ \emph{prioir} estimates for the approximation systems \eqref{approxnondiv} and \eqref{approxdiv}. Before this, we need a useful estimate, which implies that the high-frequency smallness of the approximating coefficient $a^{\epsilon}(t,x)$ is still preserved.
 
 \begin{lemma}\label{smallkeep}
 Assume $a(t,x)$ satisfies the assumptions in Theorem \ref{main} and Theorem \ref{maintwo}, then there exist a universal constant $C_{0}>0$ and a new frequency borderline $K>N+10$ (independent with $\epsilon$), such that
 \begin{equation}
 \||\p_{x}|^{s}(\mathrm{Id}-S_{K})a^{\epsilon}(t,x)\|_{L^{\frac{1}{s},\infty}_{x}}\leq C_{0}\delta_{0},\ \forall t\in[0,T].
 \end{equation}
 \end{lemma}
\begin{proof}Let $\bar{a}_{\epsilon}(t,x):=(\tilde{\eta}_{\epsilon}*_{t,x}a)(t,x)$, then $a^{\epsilon}(t,x)=\chi_{\epsilon}(\bar{a}_{\epsilon}(t,x))$.
For $j>N+10$, notice that
\begin{equation}
\begin{split}
|\Delta_{j}(\chi_{\epsilon}\circ \bar{a}_{\epsilon})|&\leq |\Delta_{j}(\chi_{\epsilon}\circ S_{j}\bar{a}_{\epsilon})|+|\Delta_{j}(\chi_{\epsilon}\circ \bar{a}_{\epsilon}-\chi_{\epsilon}\circ S_{j}\bar{a}_{\epsilon})|\\
&\lesssim 2^{-j}\mathcal{M}\left(\chi_{\epsilon}'\circ S_{j}\bar{a}_{\epsilon}\cdot \p_{x}S_{j}\bar{a}_{\epsilon}\right)+\sum_{k\geq j}\mathcal{M}(\Delta_{k}\bar{a}_{\epsilon})\\
&\lesssim 2^{-j}\mathcal{M}(\p_{x}S_{2}\bar{a}_{\epsilon})+\sum_{2\leq k <N+2} 2^{(k-j)} \mathcal{M}^{(2)}(\Delta_{k}\bar{a}_{\epsilon})+\sum_{k\geq N+2}\min\{2^{k-j},1\}\mathcal{M}^{(2)}(\Delta_{k}\bar{a}_{\epsilon})\\\
&=:\sum_{l=1}^{3}\mathcal{A}_{j,l}^{\epsilon}(t,x),
\end{split}
\end{equation}
where $\mathcal{M}^{(2)}:=\mathcal{M}\circ \mathcal{M}$.

By the Littlewood-Paley characteristic for Lorentz space (Lemma \ref{lpchar}) and the Fefferman--Stein inequality (Lemma \ref{fs}), we have
\begin{equation}
\begin{split}
&\||\p_{x}|^{s}(\mathrm{Id}-S_{K})a^{\epsilon}\|_{L^{\frac{1}{s},\infty}_{x}}\lesssim\left\|\left(\sum_{j}|\Delta_{j}(|\p_{x}|^{s}(\mathrm{Id}-S_{K})a^{\epsilon})|^{2}\right)^{\frac{1}{2}}\right\|_{L^{\frac{1}{s},\infty}_{x}}\\
&\lesssim \left\|\left(\sum_{j\geq K-2}2^{2js}|\mathcal{M}(\Delta_{j}a^{\epsilon})|^{2}\right)^{\frac{1}{2}}\right\|_{L^{\frac{1}{s},\infty}_{x}}\lesssim \left\|\left(\sum_{j\geq K-2}2^{2js}|\Delta_{j}a^{\epsilon}|^{2}\right)^{\frac{1}{2}}\right\|_{L^{\frac{1}{s},\infty}_{x}}\\
&\lesssim \sum_{l=1}^{3}\left\|\left(\sum_{j\geq K-2}|2^{js}\mathcal{A}_{j,l}^{\epsilon}|^{2}\right)^{\frac{1}{2}}\right\|_{L^{\frac{1}{s},\infty}_{x}}=:I_{1}+I_{2}+I_{3}.
\end{split}
\end{equation}
\textit{Estimation of $I_{1}$.} Clearly,
\begin{equation}
I_{1}\lesssim 2^{(s-1)K}\|\mathcal{M}(\p_{x}S_{2}\bar{a}_{\epsilon})\|_{L^{\frac{1}{s},\infty}_{x}}\lesssim 2^{(s-1)K}\big(\|S_{0}\p_{x}a\|_{L^{\frac{1}{s},\infty}_{x}}+\||\p_{x}|^{s}(\mathrm{Id}-S_{0})a\|_{L^{\frac{1}{s},\infty}_{x}}\big).
\end{equation}
Hence we can choose $K$ large enough such that $I_{1}\leq \delta_{0}$.\\
\textit{Estimation of $I_{2}$.} By H\"older's inequality, we have
\begin{equation}
\sum_{j\geq K-2}\left|2^{js}\sum_{2\leq k<N+2}2^{(k-j)}\mathcal{M}^{(2)}(\Delta_{k}\bar{a}_{\epsilon})\right|^{2}\lesssim 2^{2(1-s)(N-K)}\sum_{2\leq k<N+2}|2^{ks}\mathcal{M}^{(2)}(\Delta_{k}\bar{a}_{\epsilon})|^{2}.
\end{equation}
By applying Fefferman--Stein inequality twice,
\begin{equation}
I_{2}\lesssim 2^{(1-s)(N-K)}\||\p_{x}|^{s}(Id-S_{0})a\|_{L^{\frac{1}{s},\infty}_{x}}.
\end{equation}
We can also choose $K$ large enough such that $I_{2}\leq \delta_{0}$.\\
\textit{Estimation of $I_{3}$.}
Notice that
\begin{equation}
\sum_{j\geq K-2}|2^{js}\mathcal{A}_{j,3}^{\epsilon}|^{2}\lesssim \sum_{j\geq K-2}\left|\sum_{k\geq N+2}\min \{2^{(1-s)(k-j)},2^{-s(k-j)}\}2^{ks}\mathcal{M}^{(2)}(\Delta_{k}\bar{a}^{\epsilon})\right|^{2}.
\end{equation}
By Schur's test, 
\begin{equation}
\left(\sum_{j\geq K-2}|2^{js}\mathcal{A}_{j,3}^{\epsilon}|^{2}\right)^{\frac{1}{2}}\lesssim \left(\sum_{k\geq N+2}|2^{ks}\mathcal{M}^{(2)}(\Delta_{k}\bar{a}^{\epsilon})|^{2}\right)^{\frac{1}{2}}.
\end{equation}
Hence, by virtue of the Fefferman--Stein inequality,
\begin{equation}
I_{3}\lesssim \||\p_{x}|^{s}(\mathrm{Id}-S_{N})\bar{a}^{\epsilon}\|_{L^{\frac{1}{s},\infty}_{x}}\lesssim \delta_{0}.
\end{equation}
Hence we conclude Lemma \ref{smallkeep} by summarizng above estimates.
\end{proof}
We also need the following estimates for all frequency blocks of $a^{\epsilon}$, which is required in the estimate in \eqref{eq4.12}. For brevity we denote $\|b\|_{X_{s}}:=\|(\p_{x}S_{0})b(x)\|_{L^{\frac{1}{s},\infty}_{x}}+\||\p_{x}|^{s}(\mathrm{Id}-S_{0})b(x)\|_{L^{\frac{1}{s},\infty}_{x}}$.
\begin{lemma}\label{boundkeep}
There exists a uniform-in-$\epsilon$ constant $A>0$ such that
\begin{equation}
\begin{split}
    &\||\p_{x}|^{s}(\mathrm{Id}-S_{0})a^{\epsilon}(t,x)\|_{L^{\frac{1}{s},\infty}_{x}}\leq A\cdot\sup_{t\in[0,T]}\|a(t,x)\|_{X_{s}},\ \forall t\in[0,T],\\
    &\|\p_{x}S_{0}a^{\epsilon}(t,x)\|_{L^{\frac{1}{s},\infty}_{x}}\leq A\cdot \sup_{t\in[0,T]}(\|a(t,x)\|_{X_{s}}+\|a(t,x)\|_{X_{s}}^{2}),\ \forall t\in[0,T].
\end{split}
\end{equation}
\end{lemma}
\begin{proof}The bound
$$ \sup_{t\in[0,T]}\||\p_{x}|^{s}(\mathrm{Id}-S_{0})a^{\epsilon}(t,x)\|_{L^{\frac{1}{s},\infty}_{x}}\lesssim \sup_{t\in[0,T]}\|a(t,x)\|_{X_{s}}  $$ follows from the same argument in the proof of Lemma \ref{smallkeep}, thus we focus on the estimate of $\|(\p_{x}S_{0})a^{\epsilon}\|_{L^{\frac{1}{s},\infty}_{x}}$. 

To start, by the definition of $a^{\epsilon}$ \eqref{defapp}, we have
\begin{equation}
    \begin{split}
        \|(\p_{x}S_{0})a^{\epsilon}\|_{L^{\frac{1}{s},\infty}_{x}}&=\|S_{0}\left(\chi_{\epsilon}'(\bar{a}_{\epsilon})\p_{x}\bar{a}_{\epsilon}\right)\|_{L^{\frac{1}{s},\infty}}\\
        &\leq \|S_{0}\left(\chi_{\epsilon}'(\bar{a}_{\epsilon})\cdot(\tilde{\eta}_{\epsilon}*\p_{x}S_{2}a)\right)\|_{L^{\frac{1}{s},\infty}_{x}}+\sum_{j\geq 2}\|S_{0}\left(\chi_{\epsilon}'(\bar{a}_{\epsilon})\cdot(\tilde{\eta}_{\epsilon}*\p_{x}\Delta_{j}a)\right)\|_{L^{\frac{1}{s},\infty}_{x}}\\
        &:=A+B.
    \end{split}
\end{equation}
\textit{Estimate of $A$.} The estimate of $A$ is directly. By Bernstein's inequality and the fact that $|\chi_{\epsilon}'|\lesssim 1$,
\begin{equation}
    A\lesssim \|\p_{x}S_{2}a\|_{L^{\frac{1}{s},\infty}_{x}}\lesssim\|a\|_{X_{s}}.
\end{equation}
\textit{Estimate of $B$.} The estimate of $B$ is more involved. Due to frequency localization,
\begin{equation}\label{eq3.19}
    \begin{split}
        B&\lesssim \sum_{j\geq 2}\sum_{|q-j|\leq 1}\|\Delta_{q}(\chi_{\epsilon}'(\bar{a}_{\epsilon}))\|_{L^{\infty}_{x}}\|\p_{x}\Delta_{j}a\|_{L^{\frac{1}{s},\infty}_{x}}\\
        &\lesssim \|a\|_{X_{s}}\cdot\sum_{q\geq 1}2^{q(1-s)}\|\Delta_{q}(\chi_{\epsilon}'(\bar{a}_{\epsilon}))\|_{L^{\infty}_{x}}\\
        &\lesssim \|a\|_{X_{s}}\cdot\sum_{q\geq 1}2^{-qs}\|\p_{x}(\chi_{\epsilon}'(\bar{a}_{\epsilon}))\|_{L^{\infty}_{x}}.
    \end{split}
\end{equation}
Notice that
\begin{equation}\label{eq3.20}
    |\p_{x}(\chi_{\epsilon}'(\bar{a}_{\epsilon}))|=|\chi_{\epsilon}''(\bar{a}_{\epsilon})\cdot\p_{x}\bar{a}_{\epsilon}|\lesssim \epsilon |\p_{x}\bar{a}_{\epsilon}|.
\end{equation}
But
\begin{equation}\label{eq3.21}
\begin{split}
    |\p_{x}\bar{a}_{\epsilon}|&\leq |\tilde{\eta}_{\epsilon}*\p_{x}S_{0}a|+\sum_{0\leq k\leq -\log_{2}\epsilon}|\tilde{\eta}_{\epsilon}*\p_{x}\Delta_{k}a|+\sum_{k>-\log_{2}\epsilon}|(\p_{x}\tilde{\eta}_{\epsilon})*(\Delta_{k}a)|\\
    &\lesssim \|a\|_{X_{s}}\left(1+\sum_{0\leq k\leq -\log_{2}\epsilon}\|\tilde{\eta}_{\epsilon}\|_{L^{1}_{t}L^{\frac{1}{1-s},1}_{x}}2^{k(1-s)}+\sum_{k>-\log_{2}\epsilon}\|\p_{x}\tilde{\eta}_{\epsilon}\|_{L^{1}_{t}L^{\frac{1}{1-s},1}_{x}}2^{-ks}\right)\\
    &\lesssim \epsilon^{-1}\|a\|_{X_{s}}
\end{split}
\end{equation}
By \eqref{eq3.19},\eqref{eq3.20},\eqref{eq3.21}, we have $B\lesssim \|a\|_{X_{s}}^{2}$. Hence we finish the proof of Lemma \ref{boundkeep}.
\end{proof}

\section{Priori estimates for approximation systems}\label{prioriesti}
\subsection{Priori estimates for \eqref{approxnondiv}}
First we state the uniform-in-$\epsilon$ estimates for the non-divergence form approximation system \eqref{approxnondiv}. Throughout this subsection we assume $u_{0}(x),f(t,x),a(t,x)$ satisfy the assumptions in Theorem \ref{main}.
\begin{proposition}[\bf Energy Estimate]\label{energy}
Let $u^{\epsilon}$ be the classical solution of \eqref{approxnondiv}, then we have
\begin{equation}
\sup_{0\leq t\leq T}\int_{\R}|\p_{x}u^{\epsilon}(t,x)|^{2}dx+\int_{0}^{T}\int_{\R}|\p_{x}^{2}u^{\epsilon}(t,x)|^{2}dxdt\lesssim \|u_{0}(x)\|_{\dot{H}^{1}(\R)}^{2}+\|f\|_{L^{2}_{t,x}([0,T]\times\R)}^{2}+o_{\epsilon}(1).
\end{equation}
\end{proposition}
\begin{proof}Testing the equation \eqref{approxnondiv} by $-\p_{x}^{2}u^{\epsilon}$, we have
\begin{equation*}
\int_{\R}\p_{t}u^{\epsilon}(-\p_{x}^{2}u^{\epsilon})dx+\int_{\R}a^{\epsilon}(t,x)|\p_{x}^{2}u^{\epsilon}|^{2}dx=-\int_{\R}f^{\epsilon}\cdot \p_{x}^{2}u^{\epsilon}dx.
\end{equation*}
Integrating by parts,
\begin{equation*}
\int_{\R}\p_{t}u^{\epsilon}(-\p_{x}^{2}u^{\epsilon})dx=\frac{d}{dt}\frac{1}{2}\int_{\R}|\p_{x}u^{\epsilon}|^{2}dx.
\end{equation*}
Hence, for $t\in[0,T]$,
\begin{equation}
\frac{1}{2}\int_{\R}|\p_{x}u^{\epsilon}(t,x)|^{2}dx+\int_{0}^{t}\int_{\R}a^{\epsilon}(\tau,x)|\p_{x}^{2}u^{\epsilon}(\tau,x)|^{2}dxd\tau\leq \frac{1}{2}\|\p_{x}u_{0}^{\epsilon}\|_{L^{2}(\R)}^{2}+\|f^{\epsilon}\|_{L^{2}_{t,x}}\|\p_{x}^{2}u^{\epsilon}\|_{L^{2}_{t,x}}.
\end{equation}
We complete the proof of Proposition \ref{energy} by noting $a^{\epsilon}(t,x)\geq c_{0}$ and employing Lemma \ref{mollifierbound}.
\end{proof}
Next, we state the (non-local) higher order energy estimate.
\begin{proposition}[\bf High-Order Energy Estimate]\label{high2}
Let $u^{\epsilon}$ be the classical solution of \eqref{approxnondiv}, then there exists a universal constant $\delta_{0}>0$ such that under the assumptions in Theorem \ref{main} we have
\begin{equation*}
\sup_{0\leq t \leq T}\|\p_{x}u^{\epsilon}(t,x)\|_{H^{s}(\R)}^{2}+\int_{0}^{T}\|\p_{x}^{2}u^{\epsilon}\|_{H^{s}(\R)}^{2}dt\lesssim \|u_{0}\|_{\dot{H}^{1}\cap \dot{H}^{1+s}(\R)}^{2}+\|f\|_{L^{2}_{t}H^{s}_{x}}^{2}+o_{\epsilon}(1).
\end{equation*}
\end{proposition}
\begin{proof}
To prove Proposition \ref{high2}, we need to employ the Littlewood-Paley theory presented in Section 2. For $j\geq K+L$ ($K$ is the frequency borderline in Lemma \ref{smallkeep}, and $L\geq 10$ will be chosen later), acting the projection operator $\Delta_{j}$ on the equation \eqref{approxnondiv}, we obtain
\begin{equation*}
\p_{t}\Delta_{j}u^{\epsilon}-\Delta_{j}\left(a^{\epsilon}(t,x)\p_{x}^{2}u^{\epsilon}\right)=\Delta_{j}f^{\epsilon}(t,x).
\end{equation*}
Applying Bony's paraproduct decomposition \eqref{project},
\begin{equation}\label{com1}
\begin{split}
\Delta_{j}\left(a^{\epsilon}\p_{x}^{2}u^{\epsilon}\right)&=\sum_{|l|\leq5}\Delta_{j}\left(S_{j+l-1}a^{\epsilon}\Delta_{j+l}\p_{x}^{2}u^{\epsilon}\right)+\sum_{|l|\leq5} \Delta_{j}\left(S_{j+l-1}\p_{x}^{2}u^{\epsilon}\Delta_{j+l}a^{\epsilon}\right)\\
&\quad\quad +\sum_{k\geq j-5,|q-k|\leq 1}\Delta_{j}\left(\Delta_{k}a^{\epsilon}\Delta_{q}\p_{x}^{2}u^{\epsilon}\right)
\end{split}
\end{equation}
By virtue of \eqref{com1} and Bony's paraproduct decomposition \eqref{project}, we obtain the equation of $\Delta_{j}u^{\epsilon}$ by the following calculation of commutator:

\begin{equation}\label{commutator}
\begin{split}
&\p_{t}\Delta_{j}u^{\epsilon}-a^{\epsilon}(t,x)\p_{x}^{2}\Delta_{j}u^{\epsilon}=\Delta_{j}f^{\epsilon}+\Delta_{j}\left(a^{\epsilon}\p_{x}^{2}u^{\epsilon}\right)-a^{\epsilon}(t,x)\p_{x}^{2}\Delta_{j}u^{\epsilon}\\
&\quad=\Delta_{j}f^{\epsilon}+\sum_{|l|\leq5}\Delta_{j}\left(S_{j+l-1}a^{\epsilon}\Delta_{j+l}\p_{x}^{2}u^{\epsilon}\right)-a^{\epsilon}(t,x)\p_{x}^{2}\Delta_{j}u^{\epsilon}\\
&\quad\quad +\sum_{|l|\leq5} \Delta_{j}\left(S_{j+l-1}\p_{x}^{2}u^{\epsilon}\Delta_{j+l}a^{\epsilon}\right)+\sum_{k\geq j-5,|q-k|\leq 1}\Delta_{j}\left(\Delta_{k}a^{\epsilon}\Delta_{q}\p_{x}^{2}u^{\epsilon}\right)\\
&\quad =\Delta_{j}f^{\epsilon}+(S_{j-6}-\mathrm{Id})a^{\epsilon}\Delta_{j}\p_{x}^{2}u^{\epsilon}+\sum_{|l|\leq 5}[\Delta_{j}, S_{j-6}a^{\epsilon}]\Delta_{j+l}\p_{x}^{2}u^{\epsilon}\\
&\quad\quad +\sum_{|l|\leq5} \Delta_{j}\left(S_{j+l-1}\p_{x}^{2}u^{\epsilon}\Delta_{j+l}a^{\epsilon}\right)+\sum_{-5<l\leq 5}\sum_{q=-6}^{l-2}\Delta_{j}\left(\Delta_{j+q}a^{\epsilon}\Delta_{j+l}\p_{x}^{2}u^{\epsilon}\right)\\
&\quad\quad+\sum_{k\geq j-5,|q-k|\leq 1}\Delta_{j}\left(\Delta_{k}a^{\epsilon}\Delta_{q}\p_{x}^{2}u^{\epsilon}\right)\\
&\quad =:F_{j}+M_{1,j}+M_{2,j}+M_{3,j}+M_{4,j}+M_{5,j}.
\end{split}
\end{equation}
Testing \eqref{commutator} by $-2^{2js}\partial_{x}^{2}\Delta_{j}u^{\epsilon}$, and integrating by parts, we then obtain that
\begin{equation}\label{jj}
\begin{split}
&\frac{1}{2}\left[2^{2js}\int_{\R}|\p_{x}\Delta_{j}u^{\epsilon}|^{2}dx\right]'+2^{2js}\int_{\R}a^{\epsilon}(t,x)|\p_{x}^{2}\Delta_{j}u^{\epsilon}|^{2}dx\\
&\quad\quad =-2^{2js}\int_{\R}\p_{x}^{2}\Delta_{j}u^{\epsilon}(F_{j}+M_{1,j}+M_{2,j}+M_{3,j}+M_{4,j}+M_{5,j})dx\\
&\quad\quad=:\mathcal{F}_{j}+\sum_{l=1}^{5}\mathcal{G}_{l,j}.
\end{split}
\end{equation}
Denote
\begin{equation}
\mathcal{N}_{\Lambda,\epsilon}(t):=\sum_{j\geq \Lambda}2^{2js}\|\p_{x}^{2}\Delta_{j}u^{\epsilon}\|_{L^{2}_{x}}^{2}.
\end{equation}
It is obvious that
\begin{equation}\label{estifj}
\sum_{j\geq K+L}\mathcal{F}_{j}\leq \mathcal{N}_{K+L,\epsilon}^{\frac{1}{2}}(t)\cdot\|f^{\epsilon}\|_{H^{s}_{x}}.
\end{equation}
We next focus on the estimates of $\mathcal{G}_{l,j}$ for $l\in[1,5]\cap \mathbb{N}$.\\
\textit{Estimate of $\mathcal{G}_{1,j}$.} Choose
\begin{equation*}
s+\frac{1}{2}+\frac{1}{q}=1.
\end{equation*}
Applying H\"older's inequality and Bernstein's inequality (Lemma \ref{bernstein}),
\begin{equation}
\begin{split}
|\mathcal{G}_{1,j}|&\leq 2^{2js}\left\|(S_{j-6}-\mathrm{Id})a^{\epsilon}\right\|_{L^{\frac{1}{s},\infty}}\|\Delta_{j}\p_{x}^{2}u^{\epsilon}\|_{L^{2}}\|\Delta_{j}\p_{x}^{2}u^{\epsilon}\|_{L^{q,2}}\\
&\lesssim 2^{3js}\left(\sum_{k\geq j-6}\|\Delta_{k}a^{\epsilon}\|_{L^{\frac{1}{s},\infty}}\right)\|\Delta_{j}\p_{x}^{2}u^{\epsilon}\|_{L^{2}}^{2}\\
&\lesssim 2^{3js}\left(\sum_{k\geq j-6}2^{-sk}\||\p_{x}|^{s}(\mathrm{Id}-S_{K})a^{\epsilon}\|_{L^{\frac{1}{s},\infty}_{x}}\right)\|\Delta_{j}\p_{x}^{2}u^{\epsilon}\|_{L^{2}}^{2}\\
&\lesssim 2^{2js}\|\Delta_{j}\p_{x}^{2}u^{\epsilon}\|_{L^{2}}^{2}\cdot \||\p_{x}|^{s}(\mathrm{Id}-S_{K})a^{\epsilon}\|_{L^{\frac{1}{s},\infty}_{x}}.
\end{split}
\end{equation}
Therefore, by Lemma \ref{smallkeep},
\begin{equation}\label{estig1j}
\sum_{j\geq K+L}|\mathcal{G}_{1,j}|\lesssim \delta_{0}\cdot \mathcal{N}_{K+L,\epsilon}(t)
\end{equation}
\textit{Estimate of $\mathcal{G}_{2,j}$.} We shall need a useful commutator estimate (see Lemma 2.97 in \cite{BCD}): For $j\geq 0$,
\begin{equation}
\|[\Delta_{j},b]f\|_{L^{2}_{x}}\lesssim 2^{-j}\|\nabla_{x}b\|_{L^{\infty}_{x}}\cdot\|f\|_{L^{2}_{x}}.
\end{equation}
By virtue of \eqref{commutator}, we have
\begin{equation}
\begin{split}
|\mathcal{G}_{2,j}|&\leq \sum_{|l|\leq 5}2^{2js}\|\partial_{x}^{2}\Delta_{j}u^{\epsilon}\|_{L^{2}_{x}}\left\|[\Delta_{j}, S_{j-6}a^{\epsilon}]\Delta_{j+l}\p_{x}^{2}u^{\epsilon}\right\|_{L^{2}_{x}}\\
&\lesssim \sum_{|l|\leq 5}2^{(2s-1)j}\|\partial_{x}^{2}\Delta_{j}u^{\epsilon}\|_{L^{2}_{x}}\|\partial_{x}^{2}\Delta_{j+l}u^{\epsilon}\|_{L^{2}_{x}}\|\partial_{x} S_{j-6}a^{\epsilon}\|_{L^{\infty}_{x}}\\
\end{split}
\end{equation}
Applying Bernstein's inequality and Lemma \ref{smallkeep},
\begin{equation}\label{eq4.12}
\begin{split}
\|\partial_x S_{j-6}a^\epsilon\|_{L^\infty_x}
&\le \|\partial_x S_0 a^\epsilon\|_{L^{\frac{1}{s},\infty}_x}
+\sum_{0\le k\le K+2}2^{k(1+s)}\|\Delta_k a^\epsilon\|_{L^{\frac{1}{s},\infty}_x}\\
&\quad +\sum_{K+2<k<j-6}2^{k(1+s)}\|\Delta_k a^\epsilon\|_{L^{\frac{1}{s},\infty}_x}\\
&\lesssim 2^K+2^j\delta_0.
\end{split}
\end{equation}
Where the second inequality is followed from Lemma \ref{boundkeep} and Lemma \ref{smallkeep}.
Hence,
\begin{equation}\label{estig2j}
\sum_{j\geq K+L}|\mathcal{G}_{2,j}|\lesssim  (2^{-L}+\delta_{0})\|\p_{x}^{2}u^{\epsilon}\|_{H^{s}_{x}}^{2}.
\end{equation}
\textit{Estimate of $\mathcal{G}_{3,j}$.} By the expression of $\mathcal{G}_{3,j}$ and Cauchy--Schwarz inequality,
\begin{align*}
\sum_{j\geq K+L}|\mathcal{G}_{3,j}|&\lesssim \sum_{|l|\leq 5}\int_{\R}\left(\sum_{j\geq K+L}2^{2js}|\Delta_{j+l}a^{\epsilon}|^{2}\right)^{\frac{1}{2}}\left(\sum_{j\geq K+L}2^{2js}|\Delta_{j}^{*}\Delta_{j}\p_{x}^{2}u^{\epsilon}\cdot S_{j+l-1}\p_{x}^{2}u^{\epsilon}|^{2}\right)^{\frac{1}{2}}dx.
\end{align*}
By virtue of H\"older's inequality, and the refined Sobolev embedding $H^{s}(\R)\hookrightarrow L^{\frac{2}{1-2s},2}(\R)$, we have
\begin{equation}\label{estig3j}
\begin{split}
\sum_{j\geq K+L}|\mathcal{G}_{3,j}|&\lesssim \left\|\left(\sum_{j\geq K+L-5}2^{2js}|\Delta_{j}a^{\epsilon}|^{2}\right)^{\frac{1}{2}}\right\|_{L^{\frac{1}{s},\infty}}\left\|\mathcal{M}(\p_{x}^{2}u^{\epsilon})\left(\sum_{j\geq K+L}2^{2js}|\Delta_{j}^{*}\Delta_{j}\p_{x}^{2}u^{\epsilon}|^{2}\right)^{\frac{1}{2}}\right\|_{L^{\frac{1}{1-s},1}}\\
&\lesssim \||\p_{x}|^{s}(\mathrm{Id}-S_{K})a^{\epsilon}\|_{L^{\frac{1}{s},\infty}_{x}}\|\mathcal{M}(\p_{x}^{2}u^{\epsilon})\|_{L^{\frac{2}{1-2s},2}_{x}}\left\|\left(\sum_{j\geq K+L}2^{2js}|\mathcal{M}(\Delta_{j}\p_{x}^{2}u^{\epsilon})|^{2}\right)^{\frac{1}{2}}\right\|_{L^{2}_{x}}\\
&\lesssim \delta_{0}\|\p_{x}^{2}u^{\epsilon}\|_{H^{s}_{x}}^{2}
\end{split}
\end{equation}
\textit{Estimate of $\mathcal{G}_{4,j}$.} The estimate of $\mathcal{G}_{4,j}$ is directly. By Bernstein inequality we have
\begin{equation}\label{estig4j}
\begin{split}
\sum_{j\geq K+L}|\mathcal{G}_{4,j}|&\lesssim \sup_{j\geq K+L-6}\|\Delta_{j}a^{\epsilon}\|_{L^{\infty}_{x}}\mathcal{N}_{K+L-5,\epsilon}(t)\\
&\lesssim \||\p_{x}|^{s}(\mathrm{Id}-S_{K})a^{\epsilon}\|_{L^{\frac{1}{s},\infty}_{x}}\mathcal{N}_{K+L-5,\epsilon}(t)\lesssim \delta_{0}\|\p_{x}^{2}u^{\epsilon}\|_{H^{s}_{x}}^{2}.
\end{split}
\end{equation}
\textit{Estimate of $\mathcal{G}_{5,j}$.} Let $b_{j}:=2^{js}\|\Delta_{j}\p_{x}^{2}u^{\epsilon}\|_{L^{2}(\R)}$.Then
\begin{equation}
    \sum_{j\geq K+L}|\mathcal{G}_{5,j}|\lesssim\big(\sup_{k\geq K+L-5}\|\Delta_{k}a^{\epsilon}\|_{L^{\infty}_{x}}\big)\cdot\sum_{j\geq K+L}b_{j}\big(\sum_{q\geq j-6} 2^{s(j-q)}b_{q}\big).
\end{equation}
By Bernstein inequality and Lemma \ref{smallkeep},
\begin{equation}
    \sup_{k\geq K+L-5}\|\Delta_{k}a^{\epsilon}\|_{L^{\infty}_{x}}\lesssim \||\p_{x}|^{s}(\mathrm{Id}-S_{K})a^{\epsilon}\|_{L^{\frac{1}{s},\infty}_{x}}\leq C \delta_{0}.
\end{equation}
Hence, by Young's inequality,
\begin{equation}\label{estig5j}
    \sum_{j\geq K+L}|\mathcal{G}_{5,j}|\lesssim \delta_{0}\mathcal{N}_{K+L-6,\epsilon}(t).
\end{equation}

 Finally, combining \eqref{estifj},\eqref{estig1j},\eqref{estig2j},\eqref{estig3j},\eqref{estig4j},\eqref{estig5j} and the fact that $a^{\epsilon}\geq c_{0}$, we have
 \begin{equation}\label{finalenergy}
\begin{split}
&\sup_{0\leq t\leq T}\|(\mathrm{Id}-S_{K+L})\p_{x}u^{\epsilon}\|_{H^{s}(\R)}^{2}+c_{0}\int_{0}^{T}\mathcal{N}_{K+L,\epsilon}(t)\\
&\quad\lesssim \int_{0}^{T}\left[\mathcal{N}_{K+L,\epsilon}^{\frac{1}{2}}(t)\cdot\|f^{\epsilon}(t,\cdot)\|_{H^{s}_{x}}+(2^{-L}+\delta_{0})\|\p_{x}^{2}u^{\epsilon}\|_{H^{s}_{x}}^{2}\right]dt+\|u_{0}\|_{\dot{H}^{1+s}(\R)}^{2}
\end{split}
 \end{equation}
Hence, we conclude Proposition \ref{high2} by choosing $$\delta_{0}\ll c_{0},\ 2^{-L}\leq \delta_{0},$$ and then employing Proposition \ref{energy} to control the low-frequency part in \eqref{finalenergy}.
\end{proof}
\subsection{Priori estimates for \eqref{approxdiv}}
Similar to the  priori estimates for \eqref{approxnondiv}, we can also derive the (uniform-in-$\epsilon$) estimates for the approximation system \eqref{approxdiv}. The proofs are similar to Proposition \ref{energy} and \ref{high2}, hence we omit.
\begin{proposition}\label{energydiv}
Let $u^{\epsilon}$ be the classical solution of \eqref{approxdiv}, then we have
\begin{equation}
\sup_{0\leq t\leq T}\int_{\R}|u^{\epsilon}(t,x)|^{2}dx+\int_{0}^{T}\int_{\R}|\p_{x}u^{\epsilon}(t,x)|^{2}dxdt\lesssim \|u_{0}(x)\|_{L^{2}_{x}}^{2}+\|f\|_{L^{2}([0,T];H^{-1}_{x})}^{2}+o_{\epsilon}(1).
\end{equation}
\end{proposition}
\begin{proposition}\label{highdiv}
Let $u^{\epsilon}$ be the classical solution of \eqref{approxdiv}, then there exists a universal constant $\delta_{0}>0$ such that under the assumptions in Theorem \ref{maintwo} we have
\begin{equation}
\sup_{0\leq t \leq T}\|u^{\epsilon}(t,x)\|_{H^{s}(\R)}^{2}+\int_{0}^{T}\|\p_{x}u^{\epsilon}\|_{H^{s}(\R)}^{2}dt\lesssim \|u_{0}\|_{H^{s}(\R)}^{2}+\|f\|_{L^{2}_{t}H^{s-1}_{x}}^{2}+o_{\epsilon}(1).
\end{equation}
\end{proposition}

\section{Proof of main results}\label{pfmain}
Now we are ready to prove our main result. For brevity we only present the proof of Theorem \ref{main} in detail, since the proof of Theorem \ref{maintwo} just follows from same fashion with slight modification.\\
\textbf{ Proof of Theorem \ref{main}.} Employing Proposition \ref{energy} and Proposition \ref{high2}, there exists a constant $C>0$ such that
\begin{equation}\label{uniform}
\begin{cases}
\|\langle x \rangle^{-\frac{1}{2}}\cdot u^{\epsilon}\|_{L^{\infty}\left([0,T]\times\R\right)}&\leq C;\\
\|\p_{x}u^{\epsilon}\|_{L^{\infty}\left([0,T];H^{s}(\R)\right)}&\leq C;\\
\|\p_{xx}u^{\epsilon}\|_{L^{2}\left([0,T];H^{s}(\R)\right)}&\leq C.
\end{cases}
\end{equation}
Furthermore, for any $n\geq 1$, we have
\begin{equation}\label{localbound}
\|\p_{t}u^{\epsilon}\|_{L^{2}\left([0,T];L^{2}(I_{n})\right)}\leq Cn,\ I_{n}:=[-n,n].
\end{equation}
\begin{remark}\label{remarkfinal}
 Indeed by paraproduct \eqref{project} we can show that $\|\p_{t}u^{\epsilon}\zeta_{n}(x)\|_{L^{2}_{t}H^{s-r}_{x}}\lesssim_{n,r} 1$ for any $r>0$ and $\zeta_{n}$ defined in \eqref{defzeta}.In the proof of existence and uniqueness, \eqref{localbound} is enough for our use.
\end{remark}
To see \eqref{localbound}, by H\"older's inequality, and Sobolev's inequality,
\begin{equation}
\begin{split}
\|\p_{t}u^{\epsilon}\|_{L^{2}_{x}(I_{n})}&\lesssim \|a^{\epsilon}\|_{L^{\frac{1}{s},\infty}(I_{n})}\|\p_{xx}u^{\epsilon}\|_{L^{\frac{2}{1-2s},2}(I_{n})}+\|f^{\epsilon}\|_{L^{2}_{x}(I_{n})}\\
&\lesssim \|a^{\epsilon}\|_{L^{\frac{1}{s},\infty}(I_{n})}\|\p_{xx}u^{\epsilon}\|_{H^{s}(\R)}+\|f^{\epsilon}\|_{L^{2}_{x}(I_{n})}.\\
\end{split}
\end{equation}
Hence,
\begin{equation}
\|\p_{t}u^{\epsilon}\|_{L^{2}\left([0,T];L^{2}(I_{n})\right)}\lesssim \sup_{t\in[0,T]}\|a(t,x)\|_{L^{\frac{1}{s},\infty}_{x}(I_{n+1})}+\|f\|_{L^{2}_{t,x}([0,T]\times I_{n+1})}.
\end{equation}
Clearly,
\begin{equation}
\begin{split}
&\|(\mathrm{Id}-S_{0})a\|_{L^{\frac{1}{s},\infty}_{x}(I_{n+1})}\lesssim \||\p_{x}|^{s}(\mathrm{Id}-S_{0})a\|_{L^{\frac{1}{s},\infty}_{x}},\\
&\|S_{0}a\|_{L^{\frac{1}{s},\infty}_{x}(I_{n+1})}\lesssim n^{s}\|S_{0}a\|_{L^{\infty}_{x}(I_{n+1})}\lesssim n^{s}\left(|S_{0}a(t,0)|+n^{1-s}\|\p_{x}S_{0}a\|_{L^{\frac{1}{s},\infty}_{x}}\right).
\end{split}
\end{equation}
Hence we obtain \eqref{localbound} by employing the assumptions on $a(t,x)$.

By weak compactness of above functional spaces, there exists $u(t,x)$ also satisfies the bound  \eqref{uniform} and \eqref{localbound}, such that
\begin{equation}\label{weak}
\begin{cases}
u^{\epsilon}\rightarrow u,&\quad \text{weakly\ star\ in}\ L^{\infty}\left([0,T];L^{\infty}_{loc}(\R)\right);\\
\p_{x}u^{\epsilon}\rightarrow \p_{x}u,&\quad \text{weakly\ star\ in}\ L^{\infty}\left([0,T];H^{s}(\R)\right);\\
\p_{xx}u^{\epsilon}\rightarrow \p_{xx}u,&\quad \text{weakly in}\ L^{2}\left([0,T];H^{s}(\R)\right);\\
\p_{t}u^{\epsilon}\rightarrow \p_{t}u,&\quad\text{weakly in}\ L^{2}\left([0,T];L^{2}_{loc}(\R)\right).
\end{cases}
\end{equation}
We now divide the proof of Theorem \ref{main} into three steps.\\
\textit{Step 1.$u$ satisfies the equation \eqref{heateqnondiv}}It suffices to show that for any $\eta\in C_{c}^{\infty}\left((0,T)\times\R\right)$,
\begin{equation*}
\iint_{(0,T)\times\R} \left(\p_{t}u-a(t,x)\p_{x}^{2}u\right)\eta(t,x) dxdt=\iint_{(0,T)\times\R}f(t,x)\eta(t,x)dxdt.
\end{equation*}
Recall the equation \eqref{approxnondiv}, we only need to show that
\begin{equation*}
\lim_{\epsilon\to 0}\iint_{(0,T)\times\R}\left(\p_{t}u^{\epsilon}-a^{\epsilon}(t,x)\p_{x}^{2}u^{\epsilon}\right)\eta dxdt=\iint_{(0,T)\times\R} \left(\p_{t}u-a(t,x)\p_{x}^{2}u\right)\eta dxdt.
\end{equation*}
By weak convergence,
\begin{equation*}
\lim_{\epsilon\to 0}\iint_{(0,T)\times\R}\p_{t}u^{\epsilon}\eta dxdt=\iint_{(0,T)\times\R}\p_{t}u\eta dxdt.
\end{equation*}
For the second term, we write
\begin{equation}
\begin{split}
&\iint_{(0,T)\times\R}a^{\epsilon}(t,x)\p_{x}^{2}u^{\epsilon}\cdot \eta(t,x) dxdt-\iint_{(0,T)\times\R}a(t,x)\p_{x}^{2}u\cdot \eta(t,x) dxdt\\
&\quad=\iint_{(0,T)\times\R}\big(a^{\epsilon}(t,x)-a(t,x)\big)\p_{x}^{2}u^{\epsilon}\eta+\iint_{(0,T)\times\R}(\p_{x}^{2}u^{\epsilon}-\p_{x}^{2}u)a(t,x)\eta\\
&\quad =I+J.
\end{split}
\end{equation}
The second term tends to zero by weakly convergence of $\p_{x}^{2}u^{\epsilon}$ and the fact that $a(t,x)\in L^{\infty}_{t}L^{\frac{1}{s},\infty}_{loc}(\R)$.

For the first term, note that there exists $M>0$ such that $\eta(t,x)=0$ if $|x|\geq M$. Therefore,
\begin{equation*}
|I|\lesssim \|a^{\epsilon}(t,x)-a(t,x)\|_{L^{2}\left([0,T]\times [-M,M] \right)}\cdot \|\p_{x}^{2}u^{\epsilon}\|_{L^{2}\left([0,T]\times [-M,M] \right)}
\end{equation*}
From the definition of $a^{\epsilon}(t,x)$ \eqref{defapp} we can see that
\begin{equation*}
\lim_{\epsilon\to 0}\|a^{\epsilon}(t,x)-a(t,x)\|_{L^{2}\left([0,T]\times [-M,M] \right)}=0.
\end{equation*}
Hence $I$ tends to zero. Therefore $u$ satisfies the equation \eqref{heateqnondiv}.\\
\textit{Step 2. $u$ satisfies the initial data.} Fixed a compact subset $K\subset\R$ (see \cite{S87}), applying Aubin-Lions Lemma, we obtain that
\begin{equation*}
u^{\epsilon}\rightarrow u,\quad \text{uniformlly\ in}\ K\times [0,T]. 
\end{equation*}
For any $ x\in K $,
\begin{equation*}
\begin{split}
|u(t,x)-u_{0}(x)|&\leq|u(t,x)-u^{\epsilon}(t,x)|+|u^{\epsilon}(t,x)-u^{\epsilon}(0,x)|+|u^{\epsilon}(0,x)-u_{0}(x)|\\
&\leq \|u^{\epsilon}-u\|_{L^{\infty}(K\times [0,T])}+\|u_{0}^{\epsilon}-u_{0}\|_{L^{\infty}(K)}+|u^{\epsilon}(t,x)-u^{\epsilon}(0,x)|.
\end{split}
\end{equation*}
By the continuity of $u^{\epsilon}$,
\begin{equation*}
\limsup_{t\to 0}|u(t,x)-u_{0}(x)|\leq \|u^{\epsilon}-u\|_{L^{\infty}(K\times [0,T])}+\|u_{0}^{\epsilon}-u_{0}\|_{L^{\infty}(K)}
\end{equation*}
Letting $ \epsilon\to 0$, we obtain $u(0,x)=u_{0}(x)$.\\
\textit{Step 3: Uniqueness.} We need to show that the following equation
\begin{equation}\label{finaleq}
\begin{cases}
\p_{t}u-a(t,x)\p_{xx}u=0,\ t\in(0,T),\ x\in\R,\\
u(0,x)=0,
\end{cases}
\end{equation}
has only a trivial solution in the admissible class \eqref{weak}. Indeed, define
\begin{equation}\label{defzeta}
\zeta_{N}(x)=
\begin{cases}
1,\quad |x|\leq N,\\
0,\quad |x|\geq 2N,
\end{cases}
\quad |\p^{k}\zeta_{N}|\lesssim N^{-k}.
\end{equation}
 Testing \eqref{finaleq} by $-\p_{xx}u\zeta_{N}^{2}(x)$, we obtain the following energy inequality:
\begin{equation}
\begin{split}
& \frac{\p_t}{2}\left[\int_{\R}|\p_x u|^{2}|\zeta_N|^{2}\,dx\right]
+\int_{\R}a(t,x)|\p_{xx}u|^{2}|\zeta_N|^{2}\,dx \\
&\quad=-2\int_{\R}[(\mathrm{Id}-S_0)a]\,\p_x u\,\p_{xx}u\,\zeta_N\zeta_N'\,dx \\
&\quad\quad+\int_{\R}|\p_x u|^2
\left(\p_x S_0 a\,\zeta_{N}(\zeta_{N})'+S_0 a|(\zeta_N)'|^2+S_{0}a\zeta_N(\zeta_N)''\right)dx \\
&\quad\lesssim N^{-1}\bigg(
\|(\mathrm{Id}-S_0)a\|_{L^{\frac{1}{s},\infty}}
\|\p_x u\|_{L^{\frac{2}{1-2s},2}}
\|\p_{xx}u\|_{L^2} \\
&\quad\quad\quad +\|\p_x u\|_{L^2}^{2}\|\p_x S_0 a\|_{L^\infty}
+N^{-1}\|\p_x u\|_{L^2}^{2}\|S_0 a\|_{L^\infty(I_{2N})}
\bigg) \\
&\quad\lesssim N^{-1}\left(
\|\p_x u\|_{H^s}\|\p_{xx}u\|_{L^2}
+\|\p_x u\|_{L^2}^2
\right).
\end{split}
\end{equation}
Hence,
\begin{equation}
\sup_{0\leq t\leq T}\int_{\R}|\p_{x}u|^{2}|\zeta_{N}|^{2}\lesssim N^{-1},
\end{equation}
since $u$ belongs to the admissible class \eqref{weak}. By Fatou's lemma we obtain that $\p_{x}u$ vanishes identical, and so is $u$. Hence we finish the proof of Theorem \ref{main}.
\section*{Conflict of Interest}
The author declares no conflict of interest.
\section*{Statements and Declarations}
\paragraph{Data availability}
No datasets were generated or analysed in this study.

	\bibliographystyle{plain}
	\bibliography{heatrefer}	

\end{document}